\documentclass[12pt]{article}
\usepackage[utf8]{inputenc}
\usepackage[T1]{fontenc}
\usepackage{tikz}
\usepackage{setspace} 

\usetikzlibrary{arrows}
\usepackage{amssymb,amsthm,verbatim,mathtools}
\usepackage{amsfonts,graphicx,bm}
\usepackage{geometry, url}
\usepackage{xcolor}
\numberwithin{equation}{section}
\usepackage[authoryear]{natbib}
\newtheorem{definition}{Definition}[section]
\newtheorem{theorem}{Theorem}[section]

\newtheorem{example}{Example}[section]
\newtheorem{lemma}{Lemma}[section]
\newtheorem{corollary}{Corollary}[section]
\usepackage[pdfencoding=auto,colorlinks,linkcolor=black,citecolor=blue,urlcolor=blue]{hyperref}

\newcommand{\defsep}{\vspace{5mm}}

\usepackage{etoolbox}
\AtBeginEnvironment{equation}{\setlength{\abovedisplayskip}{5pt}}
\AtBeginEnvironment{equation}{\setlength{\belowdisplayskip}{5pt}}
\AtBeginEnvironment{equation}{\setlength{\abovedisplayshortskip}{5pt}}
\AtBeginEnvironment{equation}{\setlength{\belowdisplayshortskip}{5pt}}

\AtBeginDocument{
    \setlength{\abovedisplayskip}{5pt}  
    \setlength{\belowdisplayskip}{5pt}  
    \setlength{\abovedisplayshortskip}{5pt}  
    \setlength{\belowdisplayshortskip}{5pt}  
}

\title{Birnbaum’s principles in Venn diagrams:\\ Extended version with Lean verification}
\author{Jaime Enrique Lincovil Curivil\thanks{\footnotesize email: {\tt jlincovilc@uni.edu.pe}}
\vspace{-0.2cm}\\
     {\it \footnotesize Universidad Nacional de Ingenier\'ia}
     \vspace{-0.4cm}\\
     {\it \footnotesize Av. T\'upac Amaru 210, Lima, Per\'u}
     \vspace{0.3cm}\\
     Alexandre Galv\~ao Patriota\thanks{\footnotesize email: {\tt patriota@ime.usp.br}}
     \vspace{-0.2cm}\\
     {\it \footnotesize Departamento de Estat\'istica, IME,
     Universidade de S\~ao Paulo}
     \vspace{-0.4cm}\\
     {\it \footnotesize Rua do Mat\~ao, 1010, S\~ao Paulo/SP, 05508-090, Brazil}
     \\ \\
}
\date{}

\hypersetup{
    pdftitle={Birnbaum's principles in Venn diagrams},
    pdfauthor={Jaime Lincovil and Alexandre Galv\~ao Patriota}
}

\begin{document}
\maketitle

\begin{abstract}
Birnbaum's theorem states that the sufficiency and conditionality
principles jointly imply, and are implied by, the likelihood principle. To enable formal certification of the results in Lean, we make the conditionality relation explicit through a finite-sample formulation based on \cite{Evans2013}.
We reformulate the theorem as a statement about classes of statistical
procedures that preserve statistical relations, and show that the proof
reduces to the propagation of evidential equality along chains of
conditionality-related inference bases (experiment--observation pairs).
 The main results are illustrated by Venn diagrams, and 
  a worked finite example exhibits two pairs of inference bases: one related by conditionality but not by sufficiency, and the other related by sufficiency but not by conditionality. For finite parameter spaces
with at least two points, the class of likelihood-invariant procedures
with codomain $[0,1]$ is strictly contained in the class of
sufficiency-invariant procedures.\\

\end{abstract}

\section{Introduction}\label{introduction}

In statistical inference, sufficient and ancillary statistics are typically used to reduce the complexity of statistical models \citep{PA1997}. In addition, it is often assumed that the likelihood function contains all the experimental information about the parameter of interest \cite[page~73]{Fisher1973} and \cite[page~1]{BW1988}. Moreover, according to \cite{Cox1958}, the sample space of the statistical model should not be defined purely on power or frequentist arguments; he argued that it must also be based on the features of the experiment as actually performed. In a seminal contribution, \cite{Birnbaum1962} formalized the Likelihood (LP), Sufficiency (SP) and Conditionality (CP) principles axiomatically. He stated that  the joint adoption of the CP and the SP  leads to a mathematical characterization of the  ``evidential meaning'' by means of the likelihood function \citep[][page~271, 285] {Birnbaum1962}.  This statement is summarized by the equivalence
\begin{equation}\label{BT0}
	``\mbox{LP} \Longleftrightarrow (\mbox{CP }\& \mbox{ SP})".
\end{equation}

Some scholars have considered this a breakthrough
\citep{Basu1975, Dawid1977, BW1988, Bj1991, Bj1996}; others
raised concerns about the formulations of the principles,
common interpretations of the theorem, and the discrepancy
between Birnbaum's CP and the conditionality arguments
employed in practice
\citep{Joshi1976,Durbin1970,Holm1985,BN1995,Helland1995,Mayo2014}.
{Further analyses and reformulations of the likelihood principle
and Birnbaum's argument include
\citep{Hill1987,Joshi1989,Mayo2010,Martin2014,Gandenberger2015,Pena2017}.}
The two main ingredients of Birnbaum's formulation are:
\begin{itemize}
	\item an \textit{inference base}\footnote{\citet{Birnbaum1962} originally named the pair $(E,x)$ as an ``instance of statistical evidence''.} $(E,x)$, where $E$ is the statistical model and $x$ is the outcome of $E$, and

	\item the evidential meaning $Ev(E,x)$, which is meant to represent the essential properties of the statistical evidence contained in $(E,x)$. 

\end{itemize}

To illustrate these concepts, the following example presents four statistical models with the same parameter space $\Theta = \{1,2\}$, involving a random vector $(X_1, X_2)$ and the statistic $T(X_1, X_2) = X_1 + X_2$.

\begin{example}\label{inference_bases}

	Let $E_0$, $E_1$, $E_2$ and $E_T$ be statistical models with the same parameter space $\Theta=\{1,2 \}$, where the respective sample spaces are given by $\mathcal{X}_{E_0}= \{ 0, 1\}^2$, $\mathcal{X}_{E_1}= \{ 0, 1\}$, $\mathcal{X}_{E_2}= \{ 0, 1\}$ and $\mathcal{X}_{E_T}= \{ 0, 1, 2\}$. Here, $E_0$ is the joint probability model of two random variables $(X_1,X_2)$, $E_1$ and $E_2$ are the respective marginal models, and $E_T$ is the model induced by the statistic $T(X_1, X_2) = X_1+X_2$. The probability functions are given in the following table:

	\begin{table}[!ht]
		\centering
		\caption{The joint probability functions of $E_0$, the marginal models $E_j$, $j=1,2$, and $E_T$.}\label{tab:my-table}
		\vspace{0.5cm}
		\begin{tabular}{|c|c|c|}\hline
			$(X_1 , X_2)$ & $f_{E_0,1}$ & $f_{E_0,2}$ \\\hline
			(0,0)         & 0.1         & 0.3         \\ \hline
			(0,1)         & 0.4         & 0.2         \\ \hline
			(1,0)         & 0.4         & 0.2         \\ \hline
			(1,1)         & 0.1         & 0.3         \\ \hline
		\end{tabular} \hspace{0.5cm}
		\begin{tabular}{c}
			\begin{tabular}{|l|l|l|}\hline
				$X_1$ & $f_{E_1,1}$ & $f_{E_1,2}$ \\ \hline
				0     & 0.5         & 0.5         \\ \hline
				1     & 0.5         & 0.5         \\ \hline
			\end{tabular} \\
			\\
			\begin{tabular}{|l|l|l|}\hline
				$X_2$ & $f_{E_2,1}$ & $f_{E_2,2}$ \\ \hline
				0     & 0.5         & 0.5         \\ \hline
				1     & 0.5         & 0.5         \\ \hline
			\end{tabular}
		\end{tabular}\hspace{0.5cm}
		\begin{tabular}{|c|c|c|}\hline
			{$T(X_1,X_2)$} & $f_{E_T,1}$ & $f_{E_T,2}$ \\\hline
			0              & 0.1         & 0.3         \\ \hline
			1              & 0.8         & 0.4         \\ \hline
			2              & 0.1         & 0.3         \\ \hline
		\end{tabular}
	\end{table}

	For the joint model, all possible inference bases are $(E_0, (0,0))$, $(E_0, (0,1))$, $(E_0, (1,0))$ and $(E_0, (1,1))$. For the marginal models, the inference bases are $(E_1, 0)$, $(E_1, 1)$, $(E_2, 0)$ and $(E_2, 1)$. Finally, for the model induced by the statistic $T$, the inference bases are $(E_T, 0)$, $(E_T, 1)$ and $(E_T, 2)$. These eleven pairs are instances of inference bases whose sample and parameter spaces are finite, following the finite-model framework of \cite{Evans2013}. Note that the maximum likelihood estimates are $\widehat{\theta}_{E_0}(0,0)=2$, $\widehat{\theta}_{E_0}(0,1)=1$, $\widehat{\theta}_{E_0}(1,0)=1$ and $\widehat{\theta}_{E_0}(1,1)=2$.

	Since $X_1+X_2=0$ if and only if $(X_1,X_2)=(0,0)$, it is
	natural to expect the same inferential conclusions from
	$(E_T,0)$ and $(E_0,(0,0))$; similarly for $(E_T,2)$ and
	$(E_0,(1,1))$. For each of these four inference bases the
	maximum likelihood estimate is $2$, and the induced models of
		the MLE under $E_0$ and $E_T$ coincide. {Thus, for the
		maximum-likelihood set procedure $Ev_{\mathrm{ml}}$,}
	${Ev_{\mathrm{ml}}}\bigl(E_0,(0,0)\bigr)
		=
		{Ev_{\mathrm{ml}}}\bigl(E_0,(1,1)\bigr)
		=
		{Ev_{\mathrm{ml}}}\bigl(E_T,0\bigr)
		=
		{Ev_{\mathrm{ml}}}\bigl(E_T,2\bigr)$.
	However, maximum likelihood estimators do not always have
	identical induced models; see
	Example~\ref{Hierarchical-Example}.
\end{example}

\cite{Birnbaum1962} used inference bases, evidential meaning,
and inferential equivalence to define his statistical
principles. We introduce a \textit{statistical procedure} as a
binary relation, building on the formalization of these
principles as relations by \cite{Evans2013} (revisited in
Section~\ref{Evans-construal}). Whereas \cite{Evans2013}
framed the principles as relations on the space of inference
bases, we recast LP, CP and SP as invariance requirements on
classes of statistical procedures indexed by an arbitrary
codomain of inferential outputs. Birnbaum's theorem then
becomes a statement about the coincidence of these procedure
classes. Our framework includes Evans' setting as a special
case and allows a unified treatment of $p$-values, estimators,
likelihood functions, and confidence regions. We work with finite effective sample spaces and a common non-empty finite parameter space to obtain
procedure-level versions of Birnbaum's theorem and illustrate
them with Venn diagrams. \cite{Evans2013} diagnosed that the
proof of Birnbaum's theorem reduces to the propagation of
evidential equality along $C$-chains. Making this fully explicit through statistical procedures and Venn diagrams, we focus on whether local conditionality equivalence should be represented by equality of the values of a single global procedure $Ev$; once that representation is adopted, propagation along $C$-chains is automatic. 

Section~\ref{Evans-construal} revisits the framework of
statistical relations, with graphical representations.
Section~\ref{Class} introduces statistical procedures and their
principal instantiations. Section~\ref{remarks} concludes. The results stated here are certified in Lean~4 \citep{deMouraUllrich2021} under their stated assumptions.\footnote{\raggedright Code and verification
instructions: {\hypersetup{urlcolor=black}\url{https://github.com/AGPatriota/Birnbaum_Theorem_VennDiagrams}}.} The Lean development also certifies counterparts of all ten numbered results of \cite{Evans2013}; Lemmas~1--6, Theorems~7--9, and Corollary~10. The general relation lemmas are proved abstractly; the statistical results are proved from our definitions of $L$, $S$, and $C$ in the finite setting, under the stated parameter-space assumptions. A selected set of the formal statements, including Theorem~\ref{BT1}, Corollaries~\ref{BT2}--\ref{BT3}, the finite closure argument, the likelihood-ratio $p$-value result, and counterparts of the ten numbered results of \cite{Evans2013}, has also been verified through Palomar, with the \texttt{Challenge.lean} restricted to 1,000 lines \citep[see][for more details]{palomar-2026-09-18-000006-v1}.

\section{Evans' framework of statistical relations}\label{Evans-construal}

We begin by establishing the elementary components of our analysis, following the set-theoretic formulation of \cite{Evans2013}. A statistical model, $E$, is defined by a sample space, $\mathcal X_E$, and a family of probability measures, $\mathcal{P}_E = \{\mathbb{P}_{E,\theta} : \theta \in {\Theta}\}$, indexed by {a common non-empty parameter space $\Theta$}. We assume that both $\mathcal{X}_E$ and $\Theta$ are finite and also that $\mathcal X_E = \{ x:\max_{\theta\in\Theta} f_{E,\theta}(x)>0 \}$, where $f_{E,\theta}(x)= \mathbb{P}_{E,\theta}(\{x\})$, $x\in\mathcal{X}_E$, is the probability mass function (PMF) of the statistical model $E$. Rows with $f_{E,\theta}(x)=0$ for all $\theta$ in the conditional tables are illustrative; the actual conditional
sample space is the effective support. When an outcome $x \in \mathcal{X}_E$ is observed, the pair $I = (E, x)$ forms an \textit{inference base}. We denote by $I_i=(E_i,x_i)$ an inference base with $E_i=(\mathcal{X}_{E_i},\mathcal{P}_{E_i})$, $x_i\in \mathcal{X}_{E_i}$ and $\mathcal{P}_{E_i}=\{\mathbb{P}_{E_i,\theta}:\theta\in {\Theta}\}$.

Within a fixed set-theoretic universe, let $\mathcal I$ be the set of all inference
bases with a common non-empty finite parameter space $\Theta$
\citep{Joshi1976,Evans1985}, a finite non-empty effective sample space,
and a fixed injective outcome code $c_E:\mathcal X_E\to\mathbb N$.
The code convention for conditioning is specified below. All relations and
closures use this same $\mathcal I$, including the auxiliary models in the
proofs.

\subsection{Statistical relations}\label{statistical_relations}

As defined by \cite{Evans2013}, a statistical relation $D$ is a set containing pairs of inference bases, i.e., $D \subseteq \mathcal{I} \times \mathcal{I}$. A statistical relation contains pairs of inference bases that share a specific relationship. In what follows, we revisit the likelihood relation, the sufficiency relation, and the conditionality relation  \citep{Evans2013} and provide some examples. The first relation is the likelihood relation, which \cite{Evans2013} defined as

\defsep
\noindent 
\fbox{%
	\begin{minipage}{\dimexpr\linewidth-2\fboxsep-2\fboxrule\relax}
		{\scshape The likelihood relation} is the set $L \subseteq \mathcal{I} \times \mathcal{I}$
		defined by $(I_1,I_2)\in L$ if, and only if, there exists $ k:=k(x_1,x_2)>0$ such that
		$f_{E_1, \theta}(x_1) = k f_{E_2, \theta}(x_2)$ for every $\theta \in \Theta$.
	\end{minipage}}
\vspace*{2mm}

The set $L$ contains pairs of inference bases whose
likelihood functions are proportional to each
other. To illustrate the relationship generated by
$L$, consider the following example:

\begin{example}[Inference pairs in $L$]
	In Example \ref{inference_bases}, the inference bases $(E_0, (0,0))$ and $(E_T, 0)$ are related through $L$ since they have proportional likelihood functions $f_{E_0,\theta}(0,0) = f_{E_T, \theta}(0)$ for $\theta=1,2$; then, $\Big( (E_0, (0,0)), (E_T, 0) \Big) \in L$. Similarly, other pairs in Example \ref{inference_bases} are also in the relation $L$, including $\Big( (E_0, (1,1)), (E_T, 2) \Big)$, $\Big( (E_1, 0), (E_2, 0) \Big)$, $\Big( (E_1, 1), (E_2, 1) \Big)$, $\Big( (E_1, 0), (E_2, 1) \Big)$ and $\Big( (E_1, 1), (E_2, 0) \Big)$.
\end{example}
%

With the likelihood relation $L$
characterized, we now define the
sufficiency relation $S$, which pairs
inference bases whose minimal sufficient
statistics differ only by a bijection
\citep{Evans2013}.

\defsep
\noindent
\fbox{%
	\begin{minipage}{\dimexpr\linewidth-2\fboxsep-2\fboxrule\relax}
		{\scshape The sufficiency relation} is the set $S\subseteq\mathcal I\times\mathcal I$
		defined by $(I_1,I_2)\in S$ if and only if there exist minimal sufficient
		statistics $T_i:\mathcal X_{E_i}\twoheadrightarrow\mathcal X_{E_{T_i}}$,
		$i=1,2$, and a bijection $h:\mathcal X_{E_{T_2}}\to\mathcal X_{E_{T_1}}$
		such that $f_{E_{T_1},\theta}(h(t))=f_{E_{T_2},\theta}(t)$ for every
		$\theta\in\Theta$ and $t\in\mathcal X_{E_{T_2}}$, and
		$T_1(x_1)=h\!\left(T_2(x_2)\right)$.
		Here $\twoheadrightarrow$ means that $T_i$ maps onto the reduced sample
		space: each $t\in\mathcal X_{E_{T_i}}$ equals $T_i(x)$ for some
		$x\in\mathcal X_{E_i}$.

	\end{minipage}}
\vspace*{2mm}

Under the standing finite effective-support assumptions, the likelihood-class
statistic
\[
	T_E(x):=\{y\in\mathcal X_E:\exists k>0\ \forall\theta\in\Theta,
	f_{E,\theta}(y)=k f_{E,\theta}(x)\}
\]
is sufficient and minimal sufficient, as is every injective relabelling of
$T_E$ \citep{Basu1969,Evans2013}. Thus the sufficiency relation is well-defined.
This pointwise criterion can fail for general dominated models because
density versions may differ on parameter-dependent null sets
\citep{CavalcantePatriota2026}.

In order to illustrate the sufficiency relation, we now present examples of inference bases that satisfy
$S$.
\begin{example}[Inference pairs in $S$]\label{ex:pairs-in-S}
	Consider the statistic
	\[
		M(X_1,X_2)=\bm{1}_{\{T(X_1,X_2)=1\}},
	\]
	where $T(X_1,X_2)=X_1+X_2$ is the statistic introduced in Example~\ref{inference_bases} and $\bm{1}$ denotes the indicator function. The statistic $M$ is minimal sufficient for $E_0$, since
	\[
		M(x_1,x_2)=M(y_1,y_2)
		\iff
		{\exists k>0\ \ \forall\theta\in\Theta,\quad
		f_{E_0,\theta}(x_1,x_2)=k f_{E_0,\theta}(y_1,y_2)}.
	\]
	For the induced model $E_T$, consider the statistic
	\[
		M_T(t)=\bm{1}_{\{t=1\}}.
	\]
	Then $M_T$ is minimal sufficient for $E_T$. The induced models $E_M$ and $E_{M_T}$ are given in Table~\ref{h0T} and coincide. Therefore, taking $h$ as the identity on $\{0,1\}$, the pairs
	\[
		\begin{gathered}
			\bigl((E_0,(0,0)),(E_T,0)\bigr),\quad
			\bigl((E_0,(1,1)),(E_T,2)\bigr),\\
			\bigl((E_0,(0,1)),(E_T,1)\bigr),\quad
			\bigl((E_0,(1,0)),(E_T,1)\bigr)
		\end{gathered}
	\]
	belong to $S$. Moreover, since $M(0,0)=M(1,1)$ and $M(0,1)=M(1,0)$, the pairs
	\[
		\bigl((E_0,(0,0)),(E_0,(1,1))\bigr)
	\]
	and
	\[
		\bigl((E_0,(0,1)),(E_0,(1,0))\bigr)
	\]
	also belong to $S$, confirming explicitly the bijection required by the definition.

	\begin{table}[!ht]
		\centering
		\caption{Probability functions of the induced models $E_M$ and $E_{M_T}$.}
		\label{h0T}
		\begin{tabular}{|c|c|c|c|c|}\hline
			$m$ & $f_{E_M,1}$ & $f_{E_M,2}$ & $f_{E_{M_T},1}$ & $f_{E_{M_T},2}$ \\ \hline
			0   & 0.2         & 0.6         & 0.2             & 0.6             \\ \hline
			1   & 0.8         & 0.4         & 0.8             & 0.4             \\ \hline
		\end{tabular}
	\end{table}

\end{example}
\vspace{3mm}
In addition to the likelihood and sufficiency relations, another fundamental statistical relation is the conditionality relation defined below.

\defsep
\noindent 
\fbox{%
	\begin{minipage}{\dimexpr\linewidth-2\fboxsep-2\fboxrule\relax}
		{\scshape The conditionality relation} is the set $C \subseteq \mathcal{I}\times \mathcal{I}$ defined by $(I_1,I_2)\in C$ if and only if there exists an ancillary statistic $A:\mathcal{X}_{E_1}\to \mathcal{A}_0$ for $E_1$ such that, with $a=A(x_1)$ {and $\mathbb P_{E_1,\theta}(A=a)>0$ for every $\theta\in\Theta$}, the model $E_2$ is the conditional model of $E_1$ given $A=a$, and the observed values satisfy either $x_1=x_2$ or $x_1=(a,x_2)${, in the precise sense specified immediately below, or the same conditions hold with the roles of $I_1$ and $I_2$ reversed}.
	\end{minipage}}
\vspace*{2mm}

Precisely, there is an injection $\iota:\mathcal X_{E_2}\to\mathcal X_{E_1}$
onto the observed fiber $A^{-1}\{a\}$, with $\iota(x_2)=x_1$ and
\[
f_{E_2,\theta}(y)
=
\frac{f_{E_1,\theta}(\iota(y))}
     {\mathbb P_{E_1,\theta}(A=a)}
\qquad
\text{for every }\theta\in\Theta
\text{ and }y\in\mathcal X_{E_2}.
\]
In the ordinary branch, denoted by $x_1=x_2$, the injection preserves codes:
$c_{E_1}(\iota(y))=c_{E_2}(y)$ for every $y$. In the mixed-experiment branch,
denoted by $x_1=(a,x_2)$, $A$ is surjective onto a nontrivial $\mathcal A_0$;
no code-preservation condition is imposed. This makes precise the component
identification in Evans' auxiliary experiments \citep{Evans2013}.
Arbitrary relabelling alone is not a direct $C$-step. As $\mathbb P_{E_1,\theta}(A=a)$ does not depend on $\theta$,  the observed likelihoods are positively proportional, so $C\subseteq L$.

\medskip\noindent
Under this interpretation, $C$ subsumes the two conditionality
arguments used in the literature: the {strong} conditionality
principle (conditioning on an ancillary within a single
experiment; \cite[page~2650]{Evans2013}) and the {weak}
(mixture) conditionality principle (conditioning on the component
selector in a mixture experiment; \cite{Cox1958,Birnbaum1962}).
\citet{Kalbfleisch1975} distinguishes experimental from mathematical
ancillaries and develops a restricted experimental conditionality principle. The
role-reversal disjunct in the definition makes $C$ symmetric. Thus $C$ pairs inference bases in which
one model is obtained from the other by conditioning on an
ancillary statistic. Symmetrization is for technical convenience;
since its smallest equivalence relation is symmetric by construction, none of our
closure results depend on whether $C$ itself is taken to be
symmetric.

The definition above makes explicit, for the finite setting considered here, the conditionality relation  used in the arguments of \cite{Evans2013}. In particular, the sample-space identifications that are implicit in Evans' formulation are stated explicitly here in order to permit its formalization in Lean.

\begin{example}[Inference pairs in $C$]\label{Cond-example}
	In Example \ref{inference_bases}, let $U(X_1, X_2)=X_1$ and $V(X_1, X_2)= X_2$. As the distributions of $U$ and $V$ do not depend on $\theta$, they are ancillary to $E_0$. Denote by $E_{0|U=0}$ and $E_{0|V=0}$ the conditional models of $(X_1, X_2)$ given $U(X_1, X_2)=0$ and $V(X_1, X_2)=0$, respectively{.}
	{The two conditional models obtained from the values of $V$
	are displayed in Table~\ref{tab:Joint-Cond-V}; the $U=0$ and $V=0$
	conditional models are compared in Table~\ref{tab:Joint-Cond-V-U}.}

	By definition,
	\[
	\bigl((E_{0|U=0},(0,0)),(E_0,(0,0))\bigr)\in C
	\quad\text{and}\quad
	\bigl((E_0,(0,0)),(E_{0|V=0},(0,0))\bigr)\in C.
	\]
	However,
	\[
	\bigl((E_{0|U=0},(0,0)),(E_{0|V=0},(0,0))\bigr)\notin C.
	\]
	On each effective two-point support, every nonconstant statistic is
	injective, and the probability of either singleton depends on $\theta$.
	Thus only the trivial ancillary is available, ruling out the nontrivial
	mixed-experiment branch. The ordinary branch preserves the inherited
	outcome labels: the relabelling $(0,1)\mapsto(1,0)$, although it identifies
	the conditional models, is not a direct $C$-step.
	Thus $C$ is not transitive.

	Nevertheless, the endpoint pair belongs to $L$, because
	\[
	f_{E_{0|U=0},\theta}(0,0)
	=
	f_{E_{0|V=0},\theta}(0,0)
	\qquad
	\text{for every }\theta\in\Theta.
	\]
	
	For discussions of ancillary statistics and conditionality, see
	\cite{Basu1959,Evans2023}.

	\begin{table}[!htbp]
		\centering
		\caption{The conditional probabilities of $E_{0|V=0}$ and $E_{0|V=1}$.}
		\label{tab:Joint-Cond-V}
		\vspace{0.5cm}
		\begin{tabular}{|c|c|c|c|c|}\hline
			$(X_1 , X_2)$ & $f_{E_{0|V=0},1}$ & $f_{E_{0|V=0},2}$ & $f_{E_{0|V=1},1}$ & $f_{E_{0|V=1},2}$ \\ \hline
			(0,0)         & 0.2               & 0.6               & 0                 & 0                 \\ \hline
			(0,1)         & 0                 & 0                 & 0.8               & 0.4               \\ \hline
			(1,0)         & 0.8               & 0.4               & 0                 & 0                 \\ \hline
			(1,1)         & 0                 & 0                 & 0.2               & 0.6               \\ \hline
		\end{tabular}
	\end{table}

	\begin{table}[!htbp]
		\centering
		\caption{The conditional joint probabilities of $E_{0|U=0}$ and $E_{0|V=0}$.}
		\label{tab:Joint-Cond-V-U}
		\vspace{0.5cm}
		\begin{tabular}{|c|c|c|c|c|}\hline
			$(X_1 , X_2)$ & $f_{E_{0|U=0},1}$ & $f_{E_{0|U=0},2}$ & $f_{E_{0|V=0},1}$ & $f_{E_{0|V=0},2}$ \\\hline
			(0,0)         & 0.2               & 0.6               & 0.2               & 0.6               \\ \hline
			(0,1)         & 0.8               & 0.4               & 0                 & 0                 \\ \hline
			(1,0)         & 0                 & 0                 & 0.8               & 0.4               \\ \hline
			(1,1)         & 0                 & 0                 & 0                 & 0                 \\ \hline
		\end{tabular}
	\end{table}

	Table~\ref{tab:MLE-E-U=0} shows the induced probability functions of the maximum likelihood estimators under the joint and conditional models $E_{0}$ and $E_{0|U=0}$, respectively. The table highlights that these two induced distributions coincide.

	\begin{table}[!htbp]
		\centering
		\caption{The induced probabilities of the maximum-likelihood estimators $\widehat{\theta}_{E_0}$ and $ \widehat{\theta}_{E_{0|U=0}}$.}
		\label{tab:MLE-E-U=0}
		\vspace{0.5cm}
		\begin{tabular}{|c|c|c|c|c|}\hline
			$\widehat{\theta}$ & $f_{ \widehat{\theta}_{E_0}, 1}(\widehat{\theta})$ & $f_{ \widehat{\theta}_{E_0}, 2}(\widehat{\theta})$ & $f_{ \widehat{\theta}_{E_{0|U=0}}, 1}(\widehat{\theta})$ & $f_{ \widehat{\theta}_{E_{0|U=0}}, 2}(\widehat{\theta})$ \\ \hline
			1                  & 0.8                                                & 0.4                                                & 0.8                                                      & 0.4                                                      \\ \hline
			2                  & 0.2                                                & 0.6                                                & 0.2                                                      & 0.6                                                      \\ \hline
		\end{tabular}
	\end{table}
	{This example establishes nontransitivity for the present
	labelled-outcome relation $C$, paralleling Lemma~6 of
	\cite{Evans2013}.}
\end{example}

In Example~\ref{Cond-example}, the maximum likelihood estimates $\widehat\theta_{E_0}(0,0)=\widehat\theta_{E_{0|U=0}}(0,0)=2$
and the induced models coincide
(Tables~\ref{tab:my-table}, \ref{tab:Joint-Cond-V-U},
\ref{tab:MLE-E-U=0}). This might suggest that $C$ aligns with
the frequentist approach. However,
Example~\ref{Hierarchical-Example} shows that $C$-related bases
can share the same maximum likelihood estimates yet induce different models.

\begin{example}[Hierarchical experiment]\label{Hierarchical-Example}
	Consider a hierarchical experiment in which one first tosses a fair coin and then conducts one of two random experiments. If the coin lands on \text{``tails''} (i.e., $U=0$) with probability $1/2$, we roll a three-sided die whose statistical model is $E_3$. Conversely, if it lands on \text{``heads''} (i.e., $U=1$) with probability $1/2$, we toss another coin whose statistical model is $E_4$. Let $E^{\star}$ be the mixture model composed of $E_3$ and $E_4$ with sample space
	$\mathcal{X}_{E^{\star}}= \{0\}\times \{1,2,3 \} \cup  \{1\}\times \{0,1 \} = \{(0,1),(0,2),(0,3),(1,0),(1,1)\}$. The probability functions of $E_3$, $E_4$ and $E^{\star}$ are listed in Table~\ref{tab:hierar}, where $X_3$ denotes the number on the upward face of the die, $X_4$ is $1$ if the coin shows heads and $0$ otherwise, and $X^{\star}= (1-U) X_3 + U X_4$.

	\begin{table}[!htbp]
		\centering
		\caption{The probability functions of $E_3$, $E_4$ and $E^\star$.}
		\label{tab:hierar}
		\vspace{0.5cm}
		\footnotesize
		\setlength{\tabcolsep}{3pt}
		\begin{tabular}{ccc}
			\begin{tabular}{c}
				\begin{tabular}{|c|c|c|}\hline
					$X_3$ & $f_{E_3,1}$ & $f_{E_3,2}$ \\\hline
					1     & 1/3         & 5/12        \\ \hline
					2     & 1/3         & 2/12        \\ \hline
					3     & 1/3         & 5/12        \\ \hline
				\end{tabular} \\
				\\
				\begin{tabular}{|l|l|l|}\hline
					$X_4$ & $f_{E_4,1}$ & $f_{E_4,2}$ \\ \hline
					0     & 1/2         & 7/10        \\ \hline
					1     & 1/2         & 3/10        \\ \hline
				\end{tabular}
			\end{tabular} &
			\begin{tabular}{|c|c|c|}\hline
				$(U,X^{\star})$ & $f_{E^{\star},1}$ & $f_{E^{\star},2}$ \\\hline
				(0,1)           & 1/6               & 5/24              \\ \hline
				(0,2)           & 1/6               & 1/12              \\ \hline
				(0,3)           & 1/6               & 5/24              \\ \hline
				(1,0)           & 1/4               & 7/20              \\ \hline
				(1,1)           & 1/4               & 3/20              \\ \hline
			\end{tabular}
		\end{tabular} 
	\end{table}

	\begin{sloppypar}
		The following pairs of inference bases are related by $C$:
		$\big( (E^{\star}, (0,1)), (E_3, 1)\big)$,
		$\big( (E^{\star}, (0,2)), (E_3,2)\big)$,
		$\big( (E^{\star}, (0,3)), (E_3,3)\big)$,
		$\big( (E^{\star}, (1,0)), (E_4, 0)\big)$ and
		$\big( (E^{\star}, (1,1)), (E_4, 1)\big)$, since the inference bases $(E_3,x)$ and $(E_4,x)$ are obtained from the conditional models of $E^{\star}$ given $U=u$.
	\end{sloppypar}

	Note that the inference bases  $(E^{\star}, (0,1))$ and
	$(E_3, 1)$ yield the same maximum likelihood estimate,
	i.e., $\widehat{\theta}_{E^{\star}}(0,1)= \widehat{\theta}_{E_3}(1)=2$. However, the induced probabilities of the corresponding maximum likelihood estimators are different, as shown in Table~\ref{tab:MLE}.
	\begin{table}[!htbp]
		\centering
		\caption{The induced probabilities of the maximum likelihood estimators $\widehat{\theta}_{E^\star}$ and $\widehat{\theta}_{E_3}$.}
		\label{tab:MLE}
		\vspace{0.5cm}
		\begin{tabular}{|c|c|c|c|c|}\hline
			$\widehat{\theta}$ & $f_{ \widehat{\theta}_{E^{\star}}, 1}(\widehat{\theta})$ & $f_{ \widehat{\theta}_{E^{\star}}, 2}(\widehat{\theta})$ & $f_{ \widehat{\theta}_{E_3}, 1}(\widehat{\theta})$ & $f_{ \widehat{\theta}_{E_3}, 2}(\widehat{\theta})$ \\\hline
			1                  & 10/24                                                    & 7/30                                                     & 1/3                                                & 2/12                                               \\\hline
			2                  & 14/24                                                    & 23/30                                                    & 2/3                                                & 10/12                                              \\\hline
		\end{tabular}
	\end{table}
\end{example}

Although $C$-related bases can share the same MLE, the induced
models may differ, a phenomenon documented by
\cite{Evans1986}, \cite{Mayo2014} and isolated here by a hierarchical
construction. Frequentist procedures (e.g., confidence intervals and hypothesis tests) may then produce divergent conclusions about $\theta$
depending on the inference base. From the perspective of
\cite{Mayo2014}, such divergence is a feature, not a flaw: for
a frequentist, the sampling distribution is an essential
component of statistical evidence.

\begin{example}[Incomparability of $S$ and $C$]\label{ex:incomparability-S-C}
	The relations $S$ and $C$ are incomparable, as the definitions imply and the following pairs verify.

	First, consider the pair of inference bases
	\[
		p_A :=
		\Bigl(
		(E_0,(0,0)),(E_0,(1,1))
		\Bigr).
	\]
	By Example~\ref{ex:pairs-in-S}, the statistic
	\[
		M(X_1,X_2)=\bm{1}_{\{X_1+X_2=1\}}
	\]
	is minimal sufficient for $E_0$. Since
	\[
		M(0,0)=M(1,1)=0,
	\]
	taking the same minimal sufficient statistic on both sides and taking $h$ as the identity map gives
	\[
		p_A\in S.
	\]

	We now show that $p_A\notin C$. Since both inference bases have the same model $E_0$, a direct conditioning representation would require $E_0$ to be the conditional model of $E_0$ given an ancillary event. But all probabilities in $E_0$ are strictly positive. Therefore conditioning on a proper event changes the support, or equivalently the conditional sample space, and cannot reproduce the original model $E_0$. Hence only the trivial ancillary can reproduce $E_0$. For the trivial ancillary, the ordinary-conditioning case requires the observed values to be equal, but
	\[
		(0,0)\neq (1,1).
	\]
	{The mixed-experiment alternative is also impossible in the
	formal encoding: it would injectively parametrize a proper ancillary fiber
	of the four-point sample space by that same four-point space, contradicting
	finite cardinality. The same argument applies after reversal.}
	Thus
	\[
		p_A\in S\setminus C.
	\]

	Second, consider the pair
	\[
		p_B :=
		\Bigl(
		(E^\star,(0,1)),(E_3,1)
		\Bigr),
	\]
	where $E^\star$ and $E_3$ are the models in Example~\ref{Hierarchical-Example}. The selector $U$ is ancillary for $E^\star$, since
	\[
		\sum_{x:U(x)=0} f_{E^\star,\theta}(x)=\frac12,
		\qquad \theta=1,2.
	\]
	Moreover, conditional on $U=0$, the model $E^\star$ reduces to $E_3$, and the observed values satisfy
	\[
		x_1=(a,x_2)=(0,1),
	\]
	with $a=0$ and $x_2=1$.
	Therefore
	\[
		p_B\in C.
	\]

	However, $p_B\notin S$. The likelihood-proportionality classes, and hence the minimal sufficient partition, for $E^\star$ are
	\[
		\bigl\{
		\{(0,1),(0,3)\},
		\{(0,2)\},
		\{(1,0)\},
		\{(1,1)\}
		\bigr\}.
	\]
	Thus every minimal sufficient statistic for $E^\star$ has four induced points. For $E_3$, the minimal sufficient partition is
	\[
		\bigl\{
		\{1,3\},
		\{2\}
		\bigr\},
	\]
	so every minimal sufficient statistic for $E_3$ has two induced points. Hence there is no bijection between the induced minimal sufficient sample spaces satisfying the definition of $S$. Consequently,
	\[
		p_B\in C\setminus S.
	\]

	We conclude that
	\[
		S\not\subseteq C
		\qquad\text{and}\qquad
		C\not\subseteq S.
	\]
\end{example}
Example~\ref{ex:incomparability-S-C} shows that $S$ and $C$
are genuinely distinct local relations: neither contains the
other. This incomparability illustrates why the Venn diagram
for the relations must show $S$ and $C$ as incomparable
subsets of $L$. (Since both are reflexive, they share all
identity pairs $(I,I)$; hence $S\cap C\neq\varnothing$ even
though neither contains the other.)

We also consider a finite reduced-model version of stable
conditionality \citep{Evans2023}, denoted $SC_{\mathrm{red}}$. Stability is
assessed after reduction by a minimal sufficient statistic; the precise
definition and proofs are given in
\hyperref[reduced-stable-proof]{Appendix~\ref*{appendix}}. We obtain
$S\subseteq SC_{\mathrm{red}}$ and $SC_{\mathrm{red}}\nsubseteq C$
for every non-empty finite $\Theta$; $p_A$ illustrates the latter.
These assertions use our direct relation
$C$ and do not concern stability assessed in the original experiment.

For $D\subseteq\mathcal{I}\times\mathcal{I}$, write
$\overline{D}$ for its equivalence closure (the smallest
equivalence relation on $\mathcal{I}$ containing $D$).

For the universe $\mathcal I$ of finite inference bases with $|\Theta|\ge2$,
the relations $L$, $S$, and $C$
have the following properties, corresponding to the relation-level results
of \cite{Evans2013} and illustrated in Figure~\ref{figure10}:

\begin{itemize}
	\item $L$ and $S$ are equivalence relations;
	\item $C$ is reflexive and symmetric, but not transitive;
	\item $\overline{C}=L$;
	\item $C\cup S\subsetneq L=\overline{C\cup S}$;
	\item since $S$ is an equivalence relation, $\overline S=S\subsetneq L$;

\end{itemize}

For the strict inclusion $C\cup S\subsetneq L$, established by \cite{Evans2013}, the
$E_5/E_6$ counterexample in Example~\ref{preservs_S} provides a different explicit case and extends to any finite $\Theta$ with $|\Theta|\ge2$ by repeating its two probability columns, using each at least once. This does not assert that arbitrary models reduce to two distributions. The closure identities and $C\subsetneq L$ hold for every non-empty finite $\Theta$.


For $|\Theta|\ge2$, our Example~\ref{ex:incomparability-S-C} shows that $S\not\subseteq C$ and $C\not\subseteq S$, that is, the sufficiency relation $S$ and the conditionality relation $C$ capture distinct aspects of statistical evidence: neither contains the other, yet both sit strictly inside the likelihood relation $L$ (see Figure \ref{figure10}). More importantly, $L$ emerges naturally as the equivalence closure of $C$ alone and of $C \cup S$ together:
\[
	\overline{C} = L \qquad\text{and}\qquad \overline{C \cup S} = L.
\]
The \hyperref[finite-closure-proof]{finite closure argument in
Appendix~\ref*{appendix}} constructs a chain of four $C$-steps joining any $L$-related pair. Thus two inference bases have proportional likelihoods if and only if they can be joined by a finite chain of $C$-steps; allowing $S$-steps as well gives the same class of endpoint pairs. Figure~\ref{figure10} illustrates this relational structure.

\begin{figure}[htbp!]
	\centering
	\includegraphics[width=10cm]{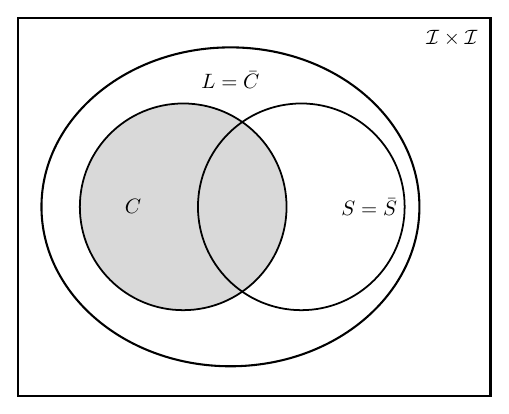}
	\caption{{Venn diagrams for the finite coded relations $L$, $S$ and $C$ used here, corresponding to Evans' framework.} $\overline{C}$ and $\overline{S}$ denote the equivalence closure of $C$ and $S$, respectively, following \cite{Evans2013}. {The proper containments and incomparability displayed here hold for $|\Theta|\ge2$.}}
	\label{figure10}
\end{figure}

\cite{Evans2013} showed that this closure property has a direct inferential consequence: equality on the local conditionality relation is mathematically equivalent to equality on the full likelihood relation. The sufficiency relation, by contrast, {when $|\Theta|\ge2$,} does not generate $L$ on its own; its equivalence closure remains strictly smaller.

While Evans' identities describe how inference bases relate to one another, statistical practice ultimately depends on the \textit{procedures} we apply to those bases. This paper shifts the focus from relations on $\mathcal{I}\times\mathcal{I}$ to the classes of procedures that respect them.  The identities above will translate, in Section~\ref{Class}, into a {nesting} of procedure classes; common methods, e.g., maximum likelihood estimators, likelihood-ratio $p$-values, Bayesian posteriors, frequentist confidence regions, then fit into this framework as concrete instances.

\subsection{Statistical principles}

The connection between Evans' relations and Birnbaum's principles can be expressed by means of the notation $Ev(\cdot)$ for the evidential output attached to each inference base. Here the notation $Ev(I)$ is used only symbolically to denote the output associated with the inference base $I$; no functional structure is being assumed at this stage. The precise relation-valued formalization is introduced in Section~\ref{Class}. With this notation, Birnbaum's LP can be written in terms of Evans' construal as follows:
\begin{equation}\label{LP_INTER}
	\mbox{LP}: \ ``\mbox{if }  (I_1,I_2)\in L, \mbox{ then } Ev(I_1)=Ev(I_2)",
\end{equation}
where equality means equality of the outputs associated with $I_1$ and $I_2$. In Section~\ref{Class}, once $Ev$ is formalized as a relation, this becomes equality of the corresponding output sets.

Analogously, the {SP and CP} can be written respectively as follows:
\begin{equation}\label{SP_INTER}
	\mbox{SP}: \ ``\mbox{if }  (I_1,I_2)\in S, \mbox{ then } Ev(I_1)=Ev(I_2)",
\end{equation}
and
\begin{equation}\label{CP_INTER}
	\mbox{CP}: \ ``\mbox{if }  (I_1,I_2)\in C, \mbox{ then } Ev(I_1)=Ev(I_2)",
\end{equation}
Furthermore, Birnbaum's theorem {\eqref{BT0}} can be reformulated in terms of Evans' relations as
\begin{equation}\label{BT0-Evans}
	{\underbrace{``\mbox{LP} \Longleftrightarrow (\mbox{CP }\& \mbox{ SP})\mbox{ for every procedure }Ev"}_{\mbox{Birnbaum}} \quad  \mbox{if, and only if, } \quad \underbrace{``L = \overline{\,C \cup S\,} \ "}_{\mbox{Evans}}.}
\end{equation}
{Here the equivalence quantifies over all relation-valued
procedures with a fixed non-empty codomain, and $L$ is an equivalence relation.}

The logical connective $``\&"$ in Birnbaum's statement imposes more restrictions on the evidential meaning than it may seem at first sight. The equality must hold not only for all pairs in $C$ or in $S$, but for all pairs in the smallest equivalence relation containing $C\cup S$. This is the closure step where transitivity enters: $\overline{C\cup S}$ is the demand that equality propagate along finite chains of direct $C$- and $S$-steps. Evans' reformulation in~{\eqref{BT0-Evans}} makes explicit a feature implicit in Birnbaum's original formulation: jointly adopting CP and SP commits us not only to the local relations but to their transitive closure.

Some authors have argued that a frequentist reading of the conditionality and sufficiency principles is more naturally formulated in terms of partial orders on evidential outputs rather than equivalence relations \citep{Holm1985,BN1995}: conditioning would then refine, rather than preserve, the evidential content of an inference base. The present framework is developed under the equivalence-based reading consistent with \cite{Birnbaum1962} and \cite{Evans2013}; the order-based variant is a natural direction for extending the present procedure-class framework and we leave its development to future work.

\section{Classes of Statistical Procedures in Evans' construal} \label{Class}

According to \cite{BW1988}, p.~197, the evidential meaning $Ev(\cdot)$ generated by a statistical procedure encompasses any collection of conclusions or reports{. In our framework, examples include} p-values, Bayes factors, s-values, e-values, confidence intervals, posterior probabilities, and so on. This collection may depend on hypotheses regarding the parameter space or on supplementary components such as prior distributions. The equality $Ev(I_1)=Ev(I_2)$ means that the entire set of conclusions for $I_1$ coincides with that for $I_2$.
{For broader treatments of likelihood, Bayesian inference, and
statistical evidence, see \citep{Pawitan2001,Robert2007,Thompson2007}; examples
of confidence-set and evidence-report constructions include
\citep{BickelPatriota2019,Patriota2013,PereiraStern}.}


In general, a procedure associates each inference base $(E,x)$ with a set of
inferential outputs or features, for example maximum-likelihood estimates,
likelihood-based summaries, $p$-values, confidence or credibility regions, or
tuples of such outputs. We formalize a statistical procedure as a binary
relation between the set of inference bases $\mathcal I$ and a fixed codomain
$\mathcal A$. When $\mathcal A=\mathcal I$, an output may itself be an
inference base, recovering Evans' relation-valued setting; ordinary functions,
such as estimators or $p$-values, are included as graph relations.

\begin{definition}\label{procedure}
	{Let $\mathcal A$ be a set.} A statistical procedure is a relation
	\[
		Ev\subseteq \mathcal I\times \mathcal A.
	\]
	For each $I\in\mathcal I$, define the output set
	\[
		Ev(I):=\{a\in\mathcal A:(I,a)\in Ev\}.
	\]
	Thus $Ev$ may be regarded as a set-valued map on $\mathcal I$. {After identifying each ordinary function $f:\mathcal I\to\mathcal A$ with its graph relation $\{(I,f(I)):I\in\mathcal I\}$, ordinary} functions are exactly those procedures for which $|Ev(I)|=1$ for every $I\in\mathcal I$. The class of all statistical procedures with codomain $\mathcal A$ is denoted by $2^{\mathcal I\times \mathcal A}$.
	
	No totality condition is imposed: the output set $Ev(I)$ is permitted
	to be empty. This convention is used in Lemma~\ref{lemma-b} and is
	why the strict inclusion in Theorem~\ref{BT1}(b) requires only
	$\mathcal A\neq\varnothing$ on the codomain.
	
\end{definition}

We now present several statistical procedures that are functions.

\begin{example}\label{thetaprocedure_estimatives}
	\normalfont

	In this example, we assume that $\Theta \subseteq \mathbb{R}$. In the case of $\mathcal A=2^{\mathbb{R}}$, examples of statistical procedures in $2^{\mathcal I\times 2^{\mathbb{R}}}$ are set-valued estimators of $\theta$ defined on $\mathcal I$. We revisit two familiar estimators as concrete instances of the framework.

	Firstly, we define the maximum likelihood set estimates for $\theta$ as the function {$Ev_{\mathrm{ml}}:\mathcal{I} \to 2^{\Theta}\subseteq 2^{\mathbb{R}}$} by
	\begin{equation}\label{MLE}
		Ev_{\mathrm{ml}}\left( E,x \right) = \left\{\theta_x \in \Theta: \ f_{E,\theta_x}(x) =
		\max_{\theta \in \Theta} f_{E,\theta}(x) \right\}.
	\end{equation}

	The set $Ev_{\mathrm{ml}}\left( E,x \right)$ {in~\eqref{MLE}} contains all values of the parameter space $\Theta$ that maximize the likelihood function $f_{E,\bullet}(x)$\footnote{Under the standing assumption that $\Theta$ is finite, the set $Ev_{\mathrm{ml}}(E,x)$ is always non-empty; in more general settings it should be interpreted as the set of all maximizers whenever they exist.}.

	Secondly, consider the class $\mathcal{H}_E$ of all unbiased estimators of $\theta$ given by:
	\[\mathcal{H}_E = \left\{h:\mathcal{X}_E\to \mathbb{R}: \sum_{x \in \mathcal{X}_E} h(x)f_{E,\theta}(x) = \theta, \ \forall \theta \in \Theta\right\}\]
	and the class

	\[
		\mathcal{T}_E =
		\left\{
		\tilde h \in \mathcal{H}_E :
		\sum_{x\in \mathcal{X}_E} (\tilde h(x)-\theta)^2 f_{E,\theta}(x)
		\le
		\sum_{x\in \mathcal{X}_E} (h(x)-\theta)^2 f_{E,\theta}(x),
		\; \forall h\in \mathcal{H}_E,\ \forall \theta\in\Theta
		\right\}.
	\]

	We define the set of uniformly minimum-variance unbiased estimates (UMVUE's) as the function $ Ev_{\mathrm{umvue}}:\mathcal{I} \to 2^{\mathbb{R}}$  by
	\begin{equation} \label{UMVUE}
		Ev_{\mathrm{umvue}}(E,x) = \left\{ h(x):  h\in \mathcal{T}_E \right\} \subseteq\mathbb R.
	\end{equation}
	{The set $Ev_{\mathrm{umvue}}(E,x)$ in~\eqref{UMVUE} contains the realized
	values $h(x)$ of all uniformly minimum-variance unbiased estimators
	$h\in\mathcal T_E$ for the model $E$.}
\end{example}

\begin{example}\label{pvalue_and_posterior}
	\normalfont
	Let $\mathcal{A}=[0,1]$. For a fixed partition of $\Theta$ given by {non-empty sets $\Theta_0$ and $\Theta_1$ with $\Theta_0\cap\Theta_1=\varnothing$ and $\Theta_0\cup\Theta_1=\Theta$}, consider the p-value function $Ev_{\text{p-value}}: \mathcal{I} \to [0,1]$ given by:
	\begin{equation*}
		Ev_{\text{p-value}}(E,x):= \max_{\theta \in \Theta_0}\sum_{y\in A_x}f_{E,\theta}(y),
	\end{equation*}
	where
	$A_x:= \{ y \in \mathcal{X}_E: \operatorname{LR}(E,y,\Theta_0)
		\leq \operatorname{LR}(E,x,\Theta_0) \}$ and
	\begin{equation}\label{LR}
		\operatorname{LR}(E,z,\Theta_0):=  \frac{\max_{\theta \in \Theta_0}f_{E,\theta}(z)}{\max_{\theta \in \Theta_1}f_{E,\theta}(z)}.
	\end{equation}
	The ratio is interpreted in the extended sense: $a/0=\infty$ for $a>0$, and $0/0=0$.

	The function $Ev_{\text{p-value}}$ is often applied to decide whether the null hypothesis $H_0:\theta \in \Theta_0$ should be rejected, for $\Theta_0 \subseteq \Theta$ and a fixed significance level $\alpha \in (0,1)$. Typically, $H_0$ is rejected when $Ev_{\text{p-value}}(E,x) < \alpha$.
	{This is the lower-tail $p$-value based on the null-to-alternative generalized
	likelihood ratio defined in~\eqref{LR}; we record this specific construction
	for use in Example~\ref{preservs_S}.}
\end{example}

\begin{example}\label{posterior-procedure}
	Let $\mathcal{A}=[0,1]$ and $\pi: \Theta \to [0,1]$ be a marginal  probability function (prior) on $\Theta$. For a fixed $\Theta_0 \subset \Theta$, consider the posterior probability function $Ev_{\pi}: \mathcal{I} \to [0,1]$ of ``$\theta \in \Theta_0$''  given by:
	\begin{equation}\label{posterior_procedure}
		Ev_{\pi}(E,x) = \sum_{\theta \in \Theta_0} \pi(\theta|x),
	\end{equation}
	where
	\begin{equation}\label{posterior_function}
		\pi(\theta \mid x) = \frac{f_{E,\theta}(x)\pi(\theta)}{\sum_{\theta \in \Theta}  f_{E,\theta}(x)\pi(\theta) },
	\end{equation}
	{The procedure in~\eqref{posterior_procedure} uses the
	posterior mass function in~\eqref{posterior_function},}
	assuming that
	\[
		\sum_{\theta\in\Theta} f_{E,\theta}(x)\pi(\theta)>0
	\]
	for every inference base $(E,x)\in\mathcal I$. This is guaranteed, for
	example, if $\pi(\theta)>0$ for every $\theta\in\Theta$ and, for every
	$(E,x)\in\mathcal I$, at least one value of $\theta$ satisfies
	$f_{E,\theta}(x)>0$.

	Note that if $(I_1, I_2) \in L$, then, for some $k>0$, we have
	\begin{equation}\label{proportional_lik_posterior}
		\pi(\theta\mid x_1)
		=
		\frac{f_{E_1,\theta}(x_1)\pi(\theta)}
		{\sum_{\theta\in\Theta} f_{E_1,\theta}(x_1)\pi(\theta)}
		=
		\frac{k\,f_{E_2,\theta}(x_2)\pi(\theta)}
		{\sum_{\theta\in\Theta} k\,f_{E_2,\theta}(x_2)\pi(\theta)}
		=
		\pi(\theta\mid x_2), \ \forall \theta\in\Theta.
	\end{equation}

	Consequently, $Ev_{\pi}(I_1)= Ev_{\pi}(I_2)$ for every $(I_1, I_2)\in L$ and a fixed $\Theta_0 \subset \Theta$. {Equality~\eqref{proportional_lik_posterior} shows this invariance, which} follows directly from the proportionality of the likelihoods \citep{BW1988} and is included as a Bayesian instance of $\mathcal{G}_L^{(\mathcal{A})}$.
\end{example}

The next examples illustrate {relation-valued statistical procedures}.

\begin{example}\citep{Evans2013} \label{Evans_proc}
	\normalfont
	Taking the codomain to be $\mathcal{A} = \mathcal{I}$, the statistical relations $L$, $S$, and $C$ are themselves examples of statistical procedures. As binary relations on $\mathcal{I}$, they are elements of the class $2^{\mathcal{I}\times \mathcal{I}}$. Thus Evans' original relations are recovered as the case $\mathcal{A}=\mathcal{I}$ of Definition~\ref{procedure}.
\end{example}

\begin{example} \label{bayes-procedure}
	Consider $\mathcal{A}= \mathcal{L}_{\Theta}\times \Pi_{\Theta}$, where $\mathcal{L}_{\Theta}$ is the set of all real-valued functions on $\Theta$ and $\Pi_\Theta$ is the {set} of probability functions on $\Theta$. The ``Bayesian inference information'' is the relation $Ev_B \subseteq \mathcal I \times \mathcal A$ defined by
	\[
		((E,x),(g,\pi)) \in Ev_B
		\quad \Longleftrightarrow \quad
		g=c\,f_{E,\bullet}(x)\ \text{for some } c>0,
		\ \text{and } \pi\in \Pi_\Theta.
	\]
	Equivalently, for each inference base $(E,x)$, the relation $Ev_B$ associates every pair
	$(c f_{E,\bullet}(x),\pi)$, where $c$ is a positive constant and $\pi$ is a prior probability function on $\Theta$.

	Note that most Bayesian inference methods based on $(E,x)$ are implemented using $(f_{E,\bullet}(x),\pi(\bullet))$.

\end{example}

\begin{example} \label{frequentist-procedure}
	Consider $\mathcal A=2^\Theta$ and fix $\alpha\in(0,1)$. The confidence-region procedure is the relation $Ev_F \subseteq \mathcal I \times \mathcal A$ defined by
	\[
		((E,x),R)\in Ev_F
	\]
	if and only if $R = CR_{E,\alpha}(x)$ for some confidence-region procedure $CR_{E,\alpha}:\mathcal X_E \to 2^\Theta$ that covers the population parameter $\theta$ with probability at least $1-\alpha$, i.e.,
	\[
		\sum_{z\in \mathcal X_E} f_{E,\theta}(z)\, \bm{1}\{\theta \in CR_{E,\alpha}(z)\} \ge 1-\alpha, \qquad \forall \theta\in\Theta.
	\]
	Thus $Ev_F$ associates each inference base $(E,x)$ with every $1-\alpha$ confidence region produced by a procedure having at least nominal coverage.
\end{example}

{The procedures $Ev_B$ and $Ev_F$ in
Examples~\ref{bayes-procedure} and~\ref{frequentist-procedure} are not required
to be functions: their output sets range over prior distributions and
confidence-region constructions, respectively.}

{When Birnbaum's principles~\eqref{LP_INTER},
\eqref{SP_INTER}, and~\eqref{CP_INTER} are read for ordinary functions,} the
expression $Ev(I_1)=Ev(I_2)$ is literal equality in the codomain. For a
relation-valued procedure, the natural analogue is equality of output sets.

\begin{definition}\label{preserving}
	Let $D\subseteq \mathcal I\times\mathcal I$. A statistical procedure $Ev\subseteq \mathcal I\times \mathcal A$ preserves $D$ if
	\[
		(I_1,I_2)\in D \quad\Longrightarrow\quad Ev(I_1)=Ev(I_2).
	\]
\end{definition}

In words, if $Ev$ preserves $D$, then any two inference bases related by $D$ generate identical collections of outputs under $Ev$.

\begin{definition}\label{joint-preserving}
	Let $\mathcal G_D^{(\mathcal A)}$ denote the subclass of $2^{\mathcal I\times\mathcal A}$ consisting of all statistical procedures that preserve $D$. We say that $Ev$ preserves jointly two relations $D_1$ and $D_2$ if
	\[
		Ev\in \mathcal G_{D_1}^{(\mathcal A)}\cap \mathcal G_{D_2}^{(\mathcal A)},
	\]
	equivalently, by Lemma~\ref{lemma-a}(b), if $Ev\in \mathcal G_{D_1\cup D_2}^{(\mathcal A)}$.
\end{definition}

Therefore, $\mathcal{G}_{D}^{(\mathcal{A})}$ contains all statistical procedures that satisfy LP, SP or CP when $D$ is instantiated as $L$, $S$ or $C$, respectively. For example, $\mathcal{G}_{L}^{(2^\Theta)}$ is the class of procedures that respect LP; analogously, $\mathcal{G}_{S}^{(2^\Theta)}$ and $\mathcal{G}_{C}^{(2^\Theta)}$ correspond to SP and CP.

The following theorem is the procedure-level analogue of Evans' closure identities: it characterizes the class of procedures that satisfy each principle and makes precise the sense in which CP and LP impose the same invariance, while SP imposes a strictly weaker one {when $|\Theta|\ge2$ and $\mathcal A\ne\varnothing$}.

Theorem~\ref{BT1} below collects the main relations among $\mathcal{G}_L^{(\mathcal{A})}$, $\mathcal{G}_C^{(\mathcal{A})}$ and $\mathcal{G}_S^{(\mathcal{A})}$; proofs are deferred to Appendix~\ref{appendix}.

\begin{theorem} \label{BT1}
	{For every non-empty finite $\Theta$ and the relations $L$, $S$, and $C$ on $\mathcal I$,} (a) $\mathcal{G}_L^{(\mathcal{A})}=\mathcal{G}_C^{(\mathcal{A})}$; (b) $\mathcal{G}_L^{(\mathcal{A})} \mathrel{{\subseteq}} \mathcal{G}_S^{(\mathcal{A})}${, with strict inclusion exactly when $|\Theta|\ge2$ and $\mathcal A\ne\varnothing$}; (c) $\mathcal{G}_{L}^{(\mathcal{A})}= \mathcal{G}_{C}^{(\mathcal{A})} \cap \mathcal{G}_{S}^{(\mathcal{A})}$.
\end{theorem}

{If $|\Theta|=1$, then $S=L$ and all three preservation classes coincide.
The strictness in part~(b) for an arbitrary non-empty codomain uses the
permission of empty output sets in Definition~\ref{procedure}.
Example~\ref{preservs_S} supplies a single-valued witness for $\mathcal A=[0,1]$.}

{The proof relies on Lemma~\ref{FT}, together with the closure
identities established in \hyperref[finite-closure-proof]{Appendix~\ref*{appendix}}
for the present finite formulation, corresponding to Theorems~7 and~9 of
\cite{Evans2013}.}

Equivalently, item~(c) follows immediately from items~(a) and~(b):
\[
	\mathcal G_C^{(\mathcal A)}\cap \mathcal G_S^{(\mathcal A)}
	=
	\mathcal G_L^{(\mathcal A)}\cap \mathcal G_S^{(\mathcal A)}
	=
	\mathcal G_L^{(\mathcal A)}.
\]
The proof through $L=\overline{C\cup S}$ is nevertheless useful because it keeps the direct connection with Birnbaum's original formulation in terms of the joint use of CP and SP.

Theorem~\ref{BT1} gives the procedure-level counterpart of Evans' relation-theoretic analysis. Item~(a) says that CP-preserving and LP-preserving procedures coincide:
\[
	\mathcal G_C^{(\mathcal A)}=\mathcal G_L^{(\mathcal A)}.
\]
This is stronger than the bare implication LP~$\Rightarrow$~CP: preserving the local relation $C$ already forces preservation of the full relation $L$. The proof of (a) decomposes into two distinct steps.
\begin{itemize}
	\item[\emph{Step (i).}] \emph{Connectivity.} If $Ev$ preserves $C$ and $(I',I'')\in L\setminus C$, {the closure identity} $\overline C=L$ provides a finite chain $I_0=I',I_1,\dots,I_n=I''$ whose consecutive pairs lie in $C$ (up to symmetry); preservation yields local equalities $Ev(I_j)=Ev(I_{j+1})$. This step uses only the definition of preservation and the symmetry of equality.

	\item[\emph{Step (ii).}] \emph{Transitivity.} The local equalities are chained to the endpoints, $Ev(I')=Ev(I'')$. This step appeals to the transitivity of the equality
	      {relation between the output sets in $2^{\mathcal{A}}$}; once a single global $Ev$ is
	      fixed, it is formally automatic. The substantive question,
	      as \cite{Evans2013} emphasized, is whether the global
	      equality-based representation of local Conditionality that
	      licenses this step is inferentially justified.
\end{itemize}

{The issue highlighted here is whether this equality-based
representation, together with the chosen relations and model universe,
is inferentially appropriate.}

{Suppose $|\Theta|\ge2$ and $\mathcal A\ne\varnothing$.} Item~(b) shows that SP is strictly weaker at the procedure level: every LP-preserving procedure preserves $S$, but not conversely. Item~(c) recovers Birnbaum's joint formulation: requiring both SP and CP is equivalent to requiring LP. Figure~\ref{figure2} depicts these relationships; note that $\mathcal G_C^{(\mathcal A)}$ is strictly contained in $\mathcal G_S^{(\mathcal A)}$, in contrast with the incomparability of $S$ and $C$ in Figure~\ref{figure10}.

\begin{figure}[htp!]
	\centering
	\includegraphics[width=10cm]{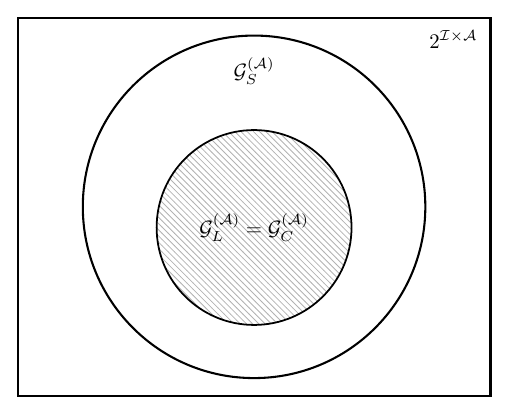}
	\caption{Venn diagram for the classes $\mathcal G_L^{(\mathcal A)}$, $\mathcal G_C^{(\mathcal A)}$ and $\mathcal G_S^{(\mathcal A)}$ (compare with Figure~\ref{figure10}, where $S$ and $C$ are incomparable). {The strict containment displayed here holds for $|\Theta|\ge2$ and $\mathcal A\neq\varnothing$.}}
	\label{figure2}
\end{figure}

To illustrate Theorem~\ref{BT1}, consider the maximum likelihood set estimates $Ev_{\mathrm{ml}}$, the Bayesian procedure $Ev_B$, and the likelihood-ratio p-value $Ev_{\text{p-value}}$, introduced in Examples~\ref{thetaprocedure_estimatives}, \ref{bayes-procedure}, and~\ref{pvalue_and_posterior}, respectively. Since $Ev_{\mathrm{ml}}$ provides equivalent set estimates for any pair of inference bases in $L$, it belongs to $\mathcal{G}_L^{(2^{\Theta})}$ and $\mathcal{G}_C^{(2^\Theta)}$, and therefore to $\mathcal{G}_S^{(2^\Theta)}$ by inclusion. Finally, since $\mathcal{G}_{L}^{(2^\Theta)}= \mathcal{G}_{S}^{(2^\Theta)} \cap \mathcal{G}_{C}^{(2^\Theta)}$, we can state that $Ev_{\mathrm{ml}}$ jointly respects both SP and CP for the case $\mathcal{A}= 2^{\Theta}$. It is well known that $Ev_B \in \mathcal{G}_{L}^{(\mathcal{L}_{\Theta}\times \Pi_{\Theta})}$ by the proportionality of the likelihood and the common parameter space. The likelihood-ratio p-value $Ev_{\text{p-value}}$ exhibits a familiar inferential output that satisfies the sufficiency principle but violates the conditionality principle (equivalently, by Theorem~\ref{BT1}(a), violates the likelihood principle). This makes the strict inclusion $\mathcal{G}_L^{(\mathcal{A})} \subsetneq \mathcal{G}_S^{(\mathcal{A})}$ concrete and illustrates why, in the debate over Birnbaum's theorem, CP carries the substantive force that SP alone does not.

On the other hand, it is well known that $Ev_{\text{p-value}}$ does not respect LP, meaning $Ev_{\text{p-value}} \notin \mathcal{G}_L^{([0,1])}$. The next example shows, however, that $Ev_{\text{p-value}} \in \mathcal{G}_S^{([0,1])}$, so the inclusion in item~(b) is strict when $\mathcal A=[0,1]$.

\begin{example}\label{preservs_S}
		{On the finite effective supports, the Neyman--Fisher factorization for a sufficient statistic $T$ has a strictly positive parameter-free factor and therefore $\operatorname{LR}(E,x,\Theta_0)=\operatorname{LR}(E_T,T(x),\Theta_0)$, where LR is defined in \eqref{LR}}.  Hence, if $(I_1,I_2)\in S$, then there exists a bijection $h : \mathcal X_{E_{T_2}} \to \mathcal X_{E_{T_1}}$ such that $E_{T_1}=E_{h\circ T_2}$ and $T_1(x_1)=h(T_2(x_2))$. For $i=1,2$, define
	\[
		B_i(t):=\{u\in \mathcal X_{E_{T_i}} : \operatorname{LR}(E_{T_i},u,\Theta_0)\le \operatorname{LR}(E_{T_i},t,\Theta_0)\}.
	\]
		{This finite factorization gives the two lower-tail reduction equalities
		$Ev_{\text{p-value}}(E_i,x_i)=
		Ev_{\text{p-value}}(E_{T_i},T_i(x_i))$. Transport
		under the displayed relabelling of the reduced experiments then gives
		the final $p$-value equality.}
	By the factorization theorem,
	\[
		\operatorname{LR}(E_i,y,\Theta_0)=\operatorname{LR}(E_{T_i},T_i(y),\Theta_0),
		\qquad y\in \mathcal X_{E_i},
	\]
	and, therefore, recalling the definition of $A_x$ from Example~\ref{pvalue_and_posterior},
	\[
		A_{x_i}=T_i^{-1}\bigl(B_i(T_i(x_i))\bigr), \qquad i=1,2.
	\]
	Since $E_{T_1}=E_{h\circ T_2}$, we have
	\[
		\operatorname{LR}(E_{T_1},h(u),\Theta_0)=\operatorname{LR}(E_{T_2},u,\Theta_0),
		\qquad u\in \mathcal X_{E_{T_2}},
	\]
	hence $B_1(T_1(x_1))=h(B_2(T_2(x_2)))$. It follows that, for every $\theta\in\Theta_0$,
	\[
		\sum_{y\in A_{x_1}} f_{E_1,\theta}(y)
		=
		\sum_{t\in B_1(T_1(x_1))} f_{E_{T_1},\theta}(t)
		=
		\sum_{u\in B_2(T_2(x_2))} f_{E_{T_2},\theta}(u)
		=
		\sum_{y\in A_{x_2}} f_{E_2,\theta}(y).
	\]
	Taking the maximum over $\theta\in\Theta_0$, we obtain
	\[
		Ev_{\text{p-value}}(I_1)=Ev_{\text{p-value}}(I_2),
	\]
	that is, $Ev_{\text{p-value}}\in \mathcal{G}_S^{([0,1])}$.

	On the other hand, let $\Theta = \{\theta_1, \theta_2\}$ and probability functions in Table \ref{tabla_exemplo}.
	\begin{table}[!htbp]
		\centering
		\caption{Probability functions for $E_5$ and $E_6$.}		\label{tabla_exemplo}
		\vspace{0.5cm}
		\begin{tabular}{|c|c|c|}\hline
			$X_5$ & $f_{E_5,\theta_1}$ & $f_{E_5,\theta_2}$ \\ \hline
			$x$   & 0.1                & 0.2                \\ \hline
			$y$   & 0.9                & 0.8                \\ \hline
		\end{tabular}
		\quad
		\begin{tabular}{|c|c|c|}\hline
			$X_6$ & $f_{E_6,\theta_1}$ & $f_{E_6,\theta_2}$ \\ \hline
			$u$   & 0.2                & 0.4                \\ \hline
			$v$   & 0.8                & 0.6                \\ \hline
		\end{tabular}

	\end{table}

	We observe $x$ in $E_5$ and $u$ in $E_6$. The likelihoods satisfy $f_{E_6,\theta}(u) = 2 f_{E_5,\theta}(x)$ for all $\theta \in \Theta$.
	Thus $\bigl((E_5,x),(E_6,u)\bigr)\in L$. In fact, this pair belongs
	to $L\setminus(S\cup C)$. In each two-point experiment the likelihood
	vectors at the two outcomes are not proportional, so every minimal
	sufficient statistic is injective. The unique bijection compatible with
	the observations pairs $x$ with $u$ and $y$ with $v$; equality of the
	induced models would then require $0.1=0.2$ under $\theta_1$. Hence the
	pair is not in $S$.

	Moreover, each experiment has only the trivial ancillary, because the
	probability of either singleton depends on $\theta$. Conditioning on the
	trivial ancillary has denominator one, so a direct $C$-step would again
	require $0.1=0.2$. The mixed-experiment branch is unavailable because
	there is no nontrivial surjective ancillary selector. Hence the pair is
	not in $C$.
	
	For $H_0\!:\theta = \theta_1$, the likelihood ratio regions are $A_x = \{x\}$ (since $\operatorname{LR}(E_5,x, \{\theta_1\})=0.1/0.2=0.5 < \operatorname{LR}(E_5,y,  \{\theta_1\})=0.9/0.8=1.125$) and $A_u = \{u\}$ (since $\operatorname{LR}(E_6,u, \{\theta_1\})=0.2/0.4=0.5 < \operatorname{LR}(E_6,v, \{\theta_1\})=0.8/0.6\approx1.333$). Hence,
	\[
		Ev_{\text{p-value}}(E_5,x) = 0.1 \neq 0.2 = Ev_{\text{p-value}}(E_6,u),
	\]
	so $Ev_{\mathrm{p\text{-}value}}
		\in
		\mathcal G_S^{([0,1])}
		\setminus
		\mathcal G_L^{([0,1])}$.
\end{example}

By Theorem~\ref{BT1}(a), $Ev_{\text{p-value}}\notin\mathcal G_L^{([0,1])}$ is equivalent to $Ev_{\text{p-value}}\notin\mathcal G_C^{([0,1])}$: preservation of $S$ alone does not force preservation of $L$. Theorem~\ref{BT1}(c) then recovers Birnbaum's joint formulation: adding preservation of $C$ to preservation of $S$ leaves exactly the class $\mathcal G_L^{(\mathcal A)}$. Since (a) already gives $\mathcal G_C^{(\mathcal A)}=\mathcal G_L^{(\mathcal A)}$, the intersection in part~(c) is algebraically redundant from the $C$-side, but conceptually useful from the $S$-side: it identifies the additional constraint that eliminates procedures, such as likelihood-ratio p-values, that preserve sufficiency while violating the likelihood principle.

\begin{corollary}\label{BT2}
	{For any non-empty finite parameter space $\Theta$ and the corresponding relations on $\mathcal I$, if} $Ev \in \mathcal{G}_{C}^{(\mathcal{A})}$, then for each $(I_1,I_2) \in L \setminus C$, it holds that
	\[
		Ev(I_1) = Ev(I_2).
	\]
\end{corollary}

Figure~\ref{figure3} illustrates Corollary~\ref{BT2} in three panels: (a) a $C$-chain whose endpoints are not $C$-related; (b) the transitive closure $\overline C$ fills the missing link, and $\overline C=L$; (c) any $C$-preserving procedure $Ev$ therefore assigns identical output to all three inference bases.

\begin{figure}[htp!]
	\centering
	\includegraphics[width=\textwidth]{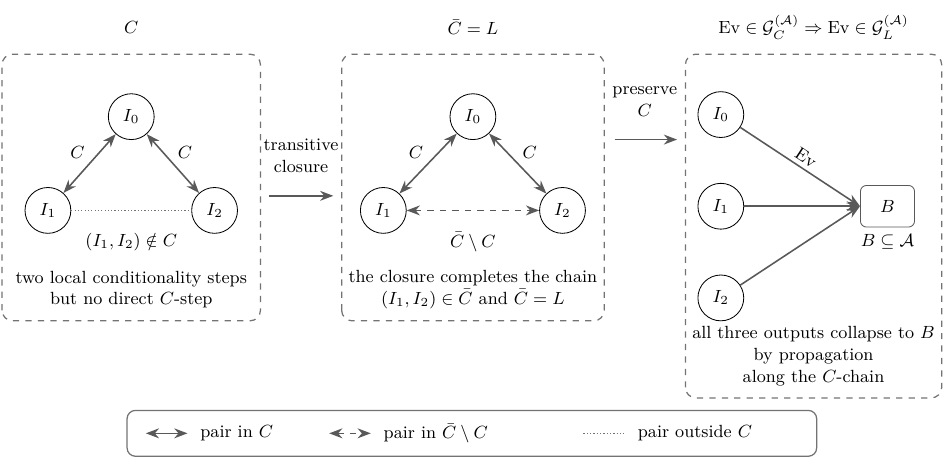}
	\caption{Three-panel illustration of the proof. (a) A $C$-chain $(I_1, I_0)\in C$ and $(I_0,I_2)\in C$; the endpoints are not directly $C$-related. (b) The transitive closure $\overline C$ completes the chain, and $\overline C=L$. (c) A $C$-preserving procedure $Ev$ assigns the same output $B$ to $I_1,I_0,I_2$: $C$-preservation forces $L$-preservation.}\label{figure3}
\end{figure}

\begin{corollary}\label{BT3}
	{For any non-empty finite parameter space $\Theta$ and the corresponding relations on $\mathcal I$, if} $Ev \in \mathcal{G}_{S}^{(\mathcal{A})} \cap \mathcal{G}_{C}^{(\mathcal{A})}$, then for each $(I_1,I_2) \in L \setminus (S \cup C)$, it holds that
	\[
		Ev(I_1) = Ev(I_2).
	\]
\end{corollary}

Corollary~\ref{BT3} is the pointwise version of
Theorem~\ref{BT1}(c): joint preservation of $S$ and $C$ forces
identical outputs even on pairs in $L\setminus(S\cup C)$
(Figure~\ref{figure10}). At the relation level, {when $|\Theta|\ge2$,} $S$ and $C$ are
incomparable (Example~\ref{ex:incomparability-S-C}); at the
procedure level, Theorem~\ref{BT1}(a)--(b) gives{, for $|\Theta|\ge2$ and $\mathcal A\ne\varnothing$,}
\[
	\mathcal{G}_C^{(\mathcal{A})}
	=
	\mathcal{G}_L^{(\mathcal{A})}
	\subsetneq
	\mathcal{G}_S^{(\mathcal{A})},
\]
so every $C$-preserving procedure preserves $S$, but not
conversely (Example~\ref{preservs_S}).

{Although Figure~\ref{figure3} depicts the pure $C$-chain used
for Corollary~\ref{BT2}, it also shows the propagation mechanism:} preservation of $C$
fixes the output on each local $C$-step; Lemma~\ref{FT}
propagates this equality along finite $C$-chains; {the closure
identity} $\overline{C}=L$ converts the propagation into
likelihood preservation. Jointly, Corollary~\ref{BT3} follows
via $\overline{S\cup C}=L$. Example~\ref{Hierarchical-Example}
makes the substantive content explicit: $C$-related bases can
share the same MLE yet induce different sampling
distributions. {The transitivity step then forces evidential
equality across distinct experimental contexts; for a frequentist who regards
the sampling distribution as evidentially relevant, that forced equality is
precisely the objection.}

The relational structure is clear: $C$ is reflexive and symmetric but not transitive; $L$ is its transitive closure. Birnbaum's theorem is the statement that any procedure treating $C$-related bases as evidentially equivalent must, by transitivity, treat all $L$-related bases as equivalent. Once a single global
procedure $Ev$ is fixed, transitivity of equality is automatic. The inferential
issue considered here is whether the global
equality-based representation of local Conditionality that
makes this propagation automatic is inferentially justified.

\section{Concluding remarks}\label{remarks}

We have reformulated Birnbaum's principles as preservation
requirements on statistical procedures
$Ev\subseteq\mathcal{I}\times\mathcal{A}$ for an arbitrary
codomain $\mathcal{A}$, yielding{, for $|\Theta|\ge2$,}
$\mathcal{G}_L^{(\mathcal{A})}
	=
	\mathcal{G}_C^{(\mathcal{A})}
	\subsetneq
	\mathcal{G}_S^{(\mathcal{A})}$
with
$\mathcal{G}_L^{(\mathcal{A})}
	=
	\mathcal{G}_C^{(\mathcal{A})}
	\cap
	\mathcal{G}_S^{(\mathcal{A})}$
for every non-empty $\mathcal{A}$. The relation-level and
procedure-level pictures diverge: $S$ and $C$ are incomparable
as relations, yet the corresponding preservation classes are
nested.

{\cite{Evans2013} established corresponding closure identities
in his relation-level framework. Appendix~\ref{appendix} proves them for the
precise finite relations used here.} We transfer these
identities to the level of procedures, recovering Birnbaum's
theorem in procedure-level form. Within this equivalence-based
framework, the joint formulation reduces to
$L=\overline{C\cup S}$. The force of the theorem lies in the
propagation of evidential equality along chains generated by
$C$ and $S$, independent of the nature of the output.

A composite report (for instance, a point estimate together with a confidence
region and a $p$-value) need not be invariant as a whole under a single
principle. Maximum-likelihood set estimates satisfy
$Ev_{\mathrm{ml}}\in\mathcal G_L^{(2^\Theta)}$; for each fixed prior and
hypothesis satisfying Example~\ref{posterior-procedure}'s positive-normalizer
assumption, $Ev_\pi\in\mathcal G_L^{([0,1])}$. The lower-tail likelihood-ratio
$p$-value in Example~\ref{preservs_S} belongs to
$\mathcal G_S^{([0,1])}\setminus\mathcal G_L^{([0,1])}$.
Confidence-region invariance depends on the construction.
For a product of two reports, preservation of both components implies
preservation of the product; the reverse implication holds when both output
sets are non-empty at every inference base. This clarifies what is at stake
when a practitioner is said to ``adopt'' the likelihood principle.

{The identity
$\mathcal G_D^{(\mathcal A)}=\mathcal G_{\overline D}^{(\mathcal A)}$
holds for any relation $D$ and does not depend on cardinality. The further
identification $\overline C=L$ is proved here for the finite models specified
above. Thus the formal core is a closure phenomenon, while its statistical
interpretation depends on the relations used.}

Several questions remain open: order-based versions of CP and
SP within the procedure-level framework; a finer mapping
between our relation $C$ and the weak and strong conditionality
principles in practice{; and}
necessary and sufficient conditions for joint invariance of
composite reports under specified principles.


\section*{Declaration on the use of generative artificial intelligence}
The authors used GPT-6 Astra to assist with
formal proofs in Lean and cross-checking the manuscript against the formalization. The correspondence with the manuscript is author-audited. The authors remain responsible for the article and its accompanying code.

\newpage
\appendix

\section{Technical results}\label{appendix}

In this appendix, we present the auxiliary results and the proofs of the lemmas and theorems discussed in the paper. The notation $\overline{D}$ for the equivalence closure of $D$ is introduced in Section~\ref{Evans-construal}. The first lemma follows from \cite[Lemma~1]{Evans2013}, applied to $D\cup\mathcal I_{id}$; the second is \cite[Chapter~2, Lemma~4.4(b)]{HT1999}. We omit their proofs. For general background on equivalence relations and their closures, see also \cite{Foldes1994}.

\begin{lemma}\label{bl:1}
	Let $\mathcal{I}_{id} \subseteq \mathcal{I}\times \mathcal{I}$, such that $(I_1,I_2) \in \mathcal{I}_{id}$ if, and only if, $I_1=I_2$.  The smallest equivalence relation that contains $D$ is the set $\bar{D} \subseteq \mathcal{I}\times \mathcal{I}$, where  $(I',I'')\in \bar{D}$ if, and only if, there exist {$n\geq 1$ and} $I_1,\ldots,I_n$ such that $I_1:=I', I_n:=I''$ and $(I_j,I_{j+1})\in D \cup \mathcal{I}_{id}$ or $(I_{j+1},I_j)\in D \cup \mathcal{I}_{id}$, for each $j \in\{1,\ldots, n-1\}$.
\end{lemma}

\begin{lemma}\label{bl:3}
	Let $[I_0]_{\bar{D}}:=\{I \in \mathcal{I}: (I_0,I)\in \bar{D} \}$, then $(I_1,I_2) \notin \bar{D}$ if, and only if, $[I_1]_{\bar{D}} \cap [I_2]_{\bar{D}}= \emptyset$.
\end{lemma}

We employed Lemma~\ref{bl:1} to attain Lemma~\ref{FT}.

The following lemma, on which the main lemmas and theorems are based, provides two basic implications among classes of preserving statistical procedures.


\begin{lemma}\label{lemma-a}
	(a)   $D_1 \subseteq D_2$ implies  $\mathcal{G}_{D_2}^{(\mathcal{A})} \subseteq  \mathcal{G}_{D_1}^{(\mathcal{A})}$ and (b) $\mathcal{G}_{D_1 \cup D_2}^{(\mathcal{A})} = \mathcal{G}_{D_1}^{(\mathcal{A})}\cap \mathcal{G}_{D_2}^{(\mathcal{A})}$.
\end{lemma}

\begin{proof}
	{The class $\mathcal{G}_{D}^{(\mathcal{A})}$ is non-empty for every $D$, since the empty relation preserves every $D$.}

	(a) Assume $D_1\subseteq D_2$ and let $Ev\in\mathcal{G}_{D_2}^{(\mathcal{A})}$.
	If $(I_1,I_2)\in D_1$, then $(I_1,I_2)\in D_2$, and by definition of $\mathcal{G}_{D_2}^{(\mathcal{A})}$ we have $Ev(I_1)=Ev(I_2)$. Hence $Ev$ preserves $D_1$, i.e. $Ev\in\mathcal{G}_{D_1}^{(\mathcal{A})}$. Consequently $\mathcal{G}_{D_2}^{(\mathcal{A})}\subseteq\mathcal{G}_{D_1}^{(\mathcal{A})}$.

	(b) We prove the two inclusions.

	Inclusion $\subseteq$:
	Take $Ev\in\mathcal{G}_{D_1\cup D_2}^{(\mathcal{A})}$. For any $(I_1,I_2)\in D_1$ we have $(I_1,I_2)\in D_1\cup D_2$, so $Ev(I_1)=Ev(I_2)$; thus $Ev$ preserves $D_1$. Similarly $Ev$ preserves $D_2$. Therefore $Ev\in\mathcal{G}_{D_1}^{(\mathcal{A})}\cap\mathcal{G}_{D_2}^{(\mathcal{A})}$.

	Inclusion $\supseteq$:
	Take $Ev\in\mathcal{G}_{D_1}^{(\mathcal{A})}\cap\mathcal{G}_{D_2}^{(\mathcal{A})}$. If $(I_1,I_2)\in D_1\cup D_2$, then either $(I_1,I_2)\in D_1$ or $(I_1,I_2)\in D_2$. In either case preservation of $D_1$ or $D_2$ yields $Ev(I_1)=Ev(I_2)$. Hence $Ev$ preserves $D_1\cup D_2$, i.e. $Ev\in\mathcal{G}_{D_1\cup D_2}^{(\mathcal{A})}$.

	Both inclusions prove the equality. \qedhere
\end{proof}

The following lemma is a stronger version of Lemma~\ref{lemma-a}(a) for equivalence relations; it deals with strict inclusion.

\begin{lemma}\label{lemma-b}
	Assume $\mathcal A\neq\varnothing$. If $\bar D_1$ and $\bar D_2$ are equivalence relations with $\bar D_1\subsetneq \bar D_2$, then
	\[
		\mathcal{G}_{\bar{D}_2}^{(\mathcal{A})} \subsetneq \mathcal{G}_{\bar{D}_1}^{(\mathcal{A})}.
	\]
\end{lemma}

\begin{proof}
	The inclusion $\mathcal{G}_{\bar{D}_2}^{(\mathcal{A})} \subseteq \mathcal{G}_{\bar{D}_1}^{(\mathcal{A})}$ follows from Lemma~\ref{lemma-a}(a), since $\bar{D}_1 \subset \bar{D}_2$.
	Choose $(I',I'') \in \bar{D}_2 \setminus \bar{D}_1$ and choose
	$a_0\in\mathcal A$. Define the relation-valued procedure $Ev'$ by
	\[
		Ev'(I):=
		\begin{cases}
			\{a_0\},     & \text{if } I \in [I']_{\bar{D}_1},    \\
			\varnothing, & \text{if } I \notin [I']_{\bar{D}_1}.
		\end{cases}
	\]
	We first show that $Ev'$ preserves $\bar D_1$. If
	$(I_1,I_2)\in\bar D_1$, then, since $\bar D_1$ is an equivalence relation,
	\[
		I_1\in[I']_{\bar D_1}
		\quad\Longleftrightarrow\quad
		I_2\in[I']_{\bar D_1}.
	\]
	Hence either both $Ev'(I_1)$ and $Ev'(I_2)$ are equal to $\{a_0\}$, or both
	are equal to $\varnothing$. Therefore $Ev'(I_1)=Ev'(I_2)$, and so
	$Ev'\in\mathcal G_{\bar D_1}^{(\mathcal A)}$.

	On the other hand, since $(I',I'')\in \bar D_2\setminus \bar D_1$, we have
	{by Lemma~\ref{bl:3} that
	$[I']_{\bar D_1}\cap[I'']_{\bar D_1}=\varnothing$, and hence}
	$I''\notin[I']_{\bar D_1}$. Thus
	\[
		Ev'(I')=\{a_0\}\neq\varnothing=Ev'(I'').
	\]
	Because $(I',I'')\in\bar D_2$, the procedure $Ev'$ does not preserve
	$\bar D_2$. Hence
	\[
		\mathcal{G}_{\bar{D}_2}^{(\mathcal{A})}
		\subsetneq
		\mathcal{G}_{\bar{D}_1}^{(\mathcal{A})}.
	\]
\end{proof}
The following result is employed to state the Birnbaum-type results discussed in the text.

\begin{lemma}\label{FT}
	$\mathcal{G}_D^{(\mathcal{A})}= \mathcal{G}_{\bar{D}}^{(\mathcal{A})}$.
\end{lemma}

\begin{proof}
	Since $D\subseteq \overline D$, Lemma~\ref{lemma-a}(a) gives
	\[
		\mathcal G_{\overline D}^{(\mathcal A)}\subseteq \mathcal G_D^{(\mathcal A)}.
	\]

	Conversely, let $Ev\in \mathcal G_D^{(\mathcal A)}$ and let $(I',I'')\in \overline D$. By Lemma~\ref{bl:1}, there exist $n\ge 1$ and inference bases $I_1,\ldots,I_n$ such that
	\[
		I_1=I',\qquad I_n=I'',
	\]
	and for each $j=1,\ldots,n-1$ either $(I_j,I_{j+1})\in D\cup \mathcal I_{id}$ or $(I_{j+1},I_j)\in D\cup \mathcal I_{id}$.

	\emph{Step (i). Connectivity.} Fix $j\in\{1,\ldots,n-1\}$. If $(I_j,I_{j+1})\in D$, then $Ev(I_j)=Ev(I_{j+1})$ because $Ev$ preserves $D$. If $(I_{j+1},I_j)\in D$, the same conclusion holds by symmetry of equality. If the pair belongs to $\mathcal I_{id}$, then $I_j=I_{j+1}$ and again $Ev(I_j)=Ev(I_{j+1})$.

	\emph{Step (ii). Transitivity.} Chaining the local equalities gives
	\[
		Ev(I_1)=Ev(I_2)=\cdots=Ev(I_n),
	\]
	so $Ev(I')=Ev(I'')$. Since $(I',I'')\in\overline D$ was arbitrary, $Ev\in \mathcal G_{\overline D}^{(\mathcal A)}$, and therefore
	\[
		\mathcal G_D^{(\mathcal A)}\subseteq \mathcal G_{\overline D}^{(\mathcal A)}.
	\]

	Combining the two inclusions yields
	\[
		\mathcal G_D^{(\mathcal A)}=\mathcal G_{\overline D}^{(\mathcal A)}.
	\]
\end{proof}

The second inclusion in the proof of Lemma~\ref{FT}
generalizes the chain argument of \citep{Birnbaum1962} and
\cite{BW1988} to an arbitrary relation $D$ and its equivalence
closure $\overline{D}$.

\paragraph{Finite closure argument.}\phantomsection\label{finite-closure-proof}
Let $I_i=(E_i,x_i)$, $i=1,2$, with $(I_1,I_2)\in L$. Write
$\ell_i(\theta)=f_{E_i,\theta}(x_i)$ and choose $k>0$ such that
$\ell_1=k\ell_2$. Set $r_1=1/(1+k)$ and $r_2=k/(1+k)$, so
$0<r_i<1$ and $r_1\ell_1=r_2\ell_2$.
Adjoin two new outcomes $u_i,v_i$ to $\mathcal X_{E_i}$ and define
$M_i$ on this disjoint union by
\[
\begin{aligned}
f_{M_i,\theta}(y)&=\frac{r_i}{2}f_{E_i,\theta}(y),
&&y\in\mathcal X_{E_i},\\
f_{M_i,\theta}(u_i)&=\frac{1-r_i\ell_i(\theta)}{2},
&\quad f_{M_i,\theta}(v_i)&=\frac{1-r_i+r_i\ell_i(\theta)}{2}.
\end{aligned}
\]
The masses sum to one and are nonnegative. Since $0\le\ell_i\le1$,
the new outcomes have positive mass for every parameter; the original
outcomes remain effective.
The events $A_i=\mathcal X_{E_i}$ and $B_i=\{x_i,u_i\}$ have
parameter-free probabilities $r_i/2$ and $1/2$, respectively.
Both events and their complements are nonempty, so their indicators
are ancillary and surjective onto $\{0,1\}$.
Conditioning on $A_i$ recovers $E_i$. Conditioning on $B_i$, identifying
$x_i$ with $0$ and $u_i$ with $1$, gives the same two-point model $H$:
\[
f_{H,\theta}(0)=r_i\ell_i(\theta),\qquad
f_{H,\theta}(1)=1-r_i\ell_i(\theta).
\]
Its first outcome is effective because $x_i$ is effective, and its second
has positive mass for every parameter. All conditioning steps use the
mixed-experiment branch of $C$; thus any injective codes for the auxiliary
models are admissible. By symmetry of $C$,
\[
\begin{aligned}
\bigl((E_1,x_1),(M_1,x_1)\bigr)&\in C,\\
\bigl((M_1,x_1),(H,0)\bigr)&\in C,\\
\bigl((H,0),(M_2,x_2)\bigr)&\in C,\\
\bigl((M_2,x_2),(E_2,x_2)\bigr)&\in C.
\end{aligned}
\]
By Lemma~\ref{bl:1}, these memberships imply
\[
\bigl((E_1,x_1),(E_2,x_2)\bigr)\in\overline C.
\]
Since the original pair was arbitrary in $L$, we obtain $L\subseteq\overline C$.
Conversely, $C\subseteq L$ and the
equivalence of $L$ imply $\overline C\subseteq L$.
Finally, $C\subseteq C\cup S\subseteq L$ gives $\overline{C\cup S}=L$.
This finite construction corresponds to the closure analysis of
\cite{Evans2013}.

\begin{proof}[Proof of Theorem \ref{BT1}]
	(a) {By $\overline C=L$, established in the
	\hyperref[finite-closure-proof]{finite closure argument above},} Lemma~\ref{FT} yields
	\[
		\mathcal G_C^{(\mathcal A)}=\mathcal G_{\overline C}^{(\mathcal A)}=\mathcal G_L^{(\mathcal A)}.
	\]

	(b) Since $S\subseteq L$, Lemma~\ref{lemma-a}(a) gives
	\[
		\mathcal G_L^{(\mathcal A)}\subseteq \mathcal G_S^{(\mathcal A)}.
	\]
	If $\mathcal A\neq\varnothing$ {and $|\Theta|\ge2$}, then $S$ and $L$ are equivalence relations and $S\subsetneq L$; therefore Lemma~\ref{lemma-b} yields
	\[
		\mathcal G_L^{(\mathcal A)}\subsetneq \mathcal G_S^{(\mathcal A)}.
	\]
	{If $|\Theta|=1$, every effective observed likelihood is
	positive, so $L=\mathcal I\times\mathcal I$; a constant statistic is
	minimal sufficient in every model, so $S=L$ and strictness is impossible.
	If $\mathcal A=\varnothing$, the empty procedure is the only procedure,
	so the preservation classes also coincide.}

	(c) By Lemma~\ref{lemma-a}(b),
	\[
		\mathcal G_C^{(\mathcal A)}\cap \mathcal G_S^{(\mathcal A)}=\mathcal G_{C\cup S}^{(\mathcal A)}.
	\]
	{By $L=\overline{C\cup S}$, established in the same
	\hyperref[finite-closure-proof]{finite closure argument},} Lemma~\ref{FT} yields
	\[
		\mathcal G_{C\cup S}^{(\mathcal A)}
		=
		\mathcal G_{\overline{C\cup S}}^{(\mathcal A)}
		=
		\mathcal G_L^{(\mathcal A)}.
	\]
	Hence
	\[
		\mathcal G_L^{(\mathcal A)}=\mathcal G_C^{(\mathcal A)}\cap \mathcal G_S^{(\mathcal A)}.
	\]
\end{proof}

\begin{proof}[Proof of Corollary \ref{BT2}]
	{By Lemma~\ref{FT} and $\overline C=L$, valid for every non-empty finite $\Theta$,} $\mathcal{G}_{C}^{(\mathcal{A})}=\mathcal{G}_{L}^{(\mathcal{A})}$. Hence, if $Ev\in\mathcal{G}_{C}^{(\mathcal{A})}$, then $Ev\in\mathcal{G}_{L}^{(\mathcal{A})}$. By Definition~\ref{preserving}, for every $(I_1,I_2)\in L$,
	\[
		\{a\in\mathcal{A}:(I_1,a)\in Ev\}
		=
		\{a\in\mathcal{A}:(I_2,a)\in Ev\}.
	\]
	In particular, this holds for every $(I_1,I_2)\in L\setminus C$.
\end{proof}

\begin{proof}[Proof of Corollary \ref{BT3}]
	{By Lemmas~\ref{lemma-a}(b) and~\ref{FT}, and
	$\overline{S\cup C}=L$, valid for every non-empty finite $\Theta$,}
	\[
		\mathcal{G}_{S}^{(\mathcal{A})} \cap \mathcal{G}_{C}^{(\mathcal{A})}
		=
		\mathcal{G}_{L}^{(\mathcal{A})}.
	\]
	Hence, if $Ev \in \mathcal{G}_{S}^{(\mathcal{A})} \cap \mathcal{G}_{C}^{(\mathcal{A})}$, then $Ev \in \mathcal{G}_{L}^{(\mathcal{A})}$. By Definition~\ref{preserving}, for every $(I_1,I_2)\in L$,
	\[
		\{a\in\mathcal{A}:(I_1,a)\in Ev\}
		=
		\{a\in\mathcal{A}:(I_2,a)\in Ev\}.
	\]
	In particular, this holds for every $(I_1,I_2)\in L\setminus (S\cup C)$.
\end{proof}

\paragraph{Reduced stable conditionality.}\phantomsection\label{reduced-stable-proof}
Within each minimal sufficient reduced model, let $\Gamma_0$ consist of
ancillary events whose intersections with every ancillary event are ancillary.
This is a finite algebra: it contains the whole space; complements follow
by subtracting intersection probabilities, and intersections by applying
stability twice. {Basu identifies $\Gamma_0$ with the laminal ancillary $\sigma$-field \citep[Definition~4 and Theorem~1, p.~250; Definition~7, p.~251; Theorem~5, p.~252]{Basu1959}.
In our finite reduced sample space, its nonempty atoms form a partition; we call these atoms laminal cells.} Each nonempty atom has positive probability for some parameter
by effectiveness, and therefore for every parameter by ancillarity.
The relation $SC_{\mathrm{red}}$ compares the observed-cell conditional PMFs,
zero outside those cells, under bijections of the full reduced sample spaces
matching the observed statistic values.

The bijection supplied by $S$ preserves all reduced-model probabilities
and intersections. It therefore transports $\Gamma_0$, its atoms and the
normalized conditional PMFs, including their zero extensions. Hence
$S\subseteq SC_{\mathrm{red}}$.
For any non-empty finite $\Theta$, consider the same coded fair-coin model
$f_{E,\theta}(0)=f_{E,\theta}(1)=1/2$ at observations $0$ and $1$.
The constant statistic is minimal sufficient, so $(E,0)$ and $(E,1)$ are
$S$-related and hence $SC_{\mathrm{red}}$-related. They are not $C$-related:
a nontrivial surjective selector has a proper observed fiber, which cannot
be in bijection with the full two-point sample space. In the ordinary
branch, preservation of the same injective code forces the identity map,
which cannot match the two observations. The argument applies in either
direction, proving $SC_{\mathrm{red}}\nsubseteq C$.
These arguments concern stability
after reduction, not stability in the original experiment.

\end{document}